\documentclass[12pt]{amsart} 
\usepackage[utf8]{inputenc}
\usepackage{amsaddr}
\usepackage[all]{xy}
\usepackage[]{graphicx}
\usepackage{color}

\usepackage{multienum}

\usepackage{hyperref}
\hypersetup{
    colorlinks=true,
    linkcolor=blue,
    filecolor=magenta,      
    urlcolor=cyan,
    pdftitle={Sharelatex Example},
    citecolor=teal,
    }

\usepackage[colorinlistoftodos]{todonotes}

\usepackage[mathlines]{lineno}

\newtheorem{thm}{Theorem}
\newtheorem{prop}[thm]{Proposition}
\newtheorem{lem}[thm]{Lemma}
\newtheorem{cor}[thm]{Corollary}

\newtheorem{rem}[thm]{Remark}

\newcommand{\R}{\mathbb{R}}
\newcommand{\T}{\mathbb{T}^2}
\newcommand{\Z}{\mathbb{Z}}

\newcommand{\Lv}{\frac{\partial L}{\partial v}}

\newcommand{\m}{\mathcal}

\newcommand{\supp}{Supp}

\begin{document}
\title[Tonelli Lagrangian systems  on $\T$]{A note on Tonelli Lagrangian systems on $\T$ with  positive topological entropy on high energy level}

\author[J. G. Damasceno]{Josu\'e G. Damasceno$^1$}
\address{Universidade Federal de Ouro Preto, R. Diogo  de Vasconcelos, 122, Pilar, Ouro Preto, MG, Brasil} 
\email{$^1$josue@ufop.edu.br}

\author[J. A. G. Miranda]{Jos\'e Ant\^onio G. Miranda$^2$}
\address{Universidade Federal de Minas Gerais, Av. Ant\^onio Carlos 6627, 31270-901, Belo Horizonte, MG, Brasil.}
\email{$^2$jan@mat.ufmg.br}

\author[L. G. Perona]{Luiz Gustavo Perona$^3$}
\address{Universidade Federal de Vi\c{c}osa, Campus Florestal, Rodovia LMG 818, km 6, 35.690-000, Florestal, MG, Brasil} 
\email{$^3$lgperona@ufv.br}


\begin{abstract}
In this work we study the dynamical behavior   Tonelli Lagrangian systems  defined on the tangent bundle of the torus   $\T=\R^2 / \Z^2$.
We prove that the  Lagrangian flow restricted to a high energy level $ E_L^{-1}(c)$ (i.e $ c> c_0(L)$) has  positive topological entropy if the flow   satisfies the Kupka-Smale propriety in   $ E_L^{-1}(c)$ (i.e, all closed orbit with energy $ c$ are hyperbolic  or elliptic and all heteroclinic intersections  are transverse on $ E_L^{-1}(c)$).
The proof requires the use of well-known results in Aubry-Mather's Theory. 
\end{abstract}

\maketitle


\section{Introduction}

Let $ \T=\R^2 / \Z^2 $ endowed  a Riemannian metric $ \langle \cdot, \cdot \rangle$. 
A Tonelli Lagrangian  on $ \T$ is a  smooth function $L: T\T \to \R$  that   satisfies  the two conditions:
\begin{itemize}
\item {\it convexity}:
 for each fiber $ T_x \T \cong \R^2$, the restriction $
L(x,\cdot):T_x \T \to \R $ has positive defined  Hessian, 
\item {\it superlinearity}:  
$\lim_{\| v\|\rightarrow \infty} \frac{L(x,v)}{\|v\|} = \infty ,$
uniformly  in $ x \in \T$.
\end{itemize}
 
The action of $L$ over an absolutely continuous curve $\gamma :[a,b] \to \T$
 is defined by
 \[A_L(\gamma) = \int_a^b L(\gamma(t),\dot \gamma(t))\ dt.\]
The  extremal  curves  for the action are given by  solutions of the
{\it Euler-Lagrange equation} which  in local coordinates can be written as:
\[\frac{\partial L}{ \partial x}  - \frac{d}{dt }  \frac{\partial L}{\partial v} =0.
\]
The  {\it Lagrangian flow}  $ \phi_t:T\T \to T\T $ is defined by $ \phi_t( x,v) = (\gamma(t),\dot{\gamma}(t))$ where $ \gamma:\R \to \T $ is the solution of {\it Euler-Lagrange equation}, 
with the initial conditions   $ \gamma(0)=x$ and $\dot{\gamma}(0)=v.$

The  {\it energy function} $  E_L : T\T \rightarrow \R $ is defined by
\begin{equation}
\label{energy function}
 E_{L}(x,v)=\left\langle \frac{\partial  L}{\partial v} (x,v)\ ,\  v \ \right\rangle -L(x,v).
\end{equation}
The  subsets   $ E_L^{-1}(c) \subset T\T $  are called {\it energy levels } and they
 are invariant by the Lagrangian flow of $ L$.   
Note that  the superlinearity condition implies that any  non-empty  energy levels  are  compact.
 Therefore  the flow $ \phi_t$  is defined for all $ t \in \R$. 
 
 The Lagrangian flow of $ L $  is conjugated to a  Hamiltonian flow on $ T^*\T $, with the canonical symplectic structure, by the Legendre transformation  $ \mathcal L : T\T \rightarrow T^*\T $ given by:
\[ \mathcal L (x,v) = \left( x, \Lv (x,v)\right).\]
The corresponding Hamiltonian  $ H: T^*\T \rightarrow \R $ is
\[ H(x,p) = \max_{v \in T_x\T} \{ p(v)- L(x,v)\}\]
 and  we have the Fenchel inequality 
\[p(v) \leq H(x,p) + L(x,v)\]
with equality, if only if, $ (x,p)=\m L( x,v)
 $ or equivalently $ p=\frac{\partial L}{\partial v}(x,v) \in T_x^*\T$.   Therefore, by \eqref{energy function},  \[ H\left(x, \frac{\partial L}{\partial v}(x,v) \right) = E (x,v).\]  
 Given a nonempty energy level $ E_L^{-1}(c) $, the  set $ H^{-1}(c) := \m L \left(  E_L^{-1}(c) \right) \subset T^* \T $ is called Hamiltonian level.

 \bigskip

We denote by $ h_{top}(L,c)$ the topological entropy of the Lagrangian flow  $ \left. \phi_t \right|_{E_L^{-1}(c)} $, for any nonempty energy level $ E_L^{-1}(c)$. 
The {\it topological entropy}, is a invariant that, roughly speaking, measures its orbits structure  complexity.  The relevant question about the topological entropy is whether it is positive or vanishes.  Namely, if $\theta\in E_L^{-1}(c)$ and $T,\delta>0$, we define the  $(\delta,T)-${\it dynamical ball} centered at  $\theta$ as 
\[B(\theta, \delta, T)=\{v\in E_L^{-1}(c)\,:\,d(\phi_t(v),\phi_t(\theta))<\delta\mbox{ for all }t\in[0,T]\},\]
 where $d$ is the distance function in $E_L^{-1}(c)$. 
 Let $N_{\delta}(T)$ be the minimal quantity of $(\delta,T)-${\it dynamical ball} needed to cover $E_L^{-1}(c)$. 
 The topological entropy is the limit  $\delta\to 0$ of the exponential growth rate of  $N_\delta(T)$, that is \[h_{top}(L,c):=\displaystyle\lim_{\delta\to 0}\limsup_{T\to\infty}\frac{1}{T}\log N_{\delta}(T).\] Thus, if $h_{top}(L,c)>0$, some dynamical balls must contract exponentially at least in one direction.
\bigskip

 For example, if  $ \langle \cdot, \cdot \rangle$ denotes  the flat metric and   $ L(x,v) =1/2  \langle v,v \rangle $, then  the corresponding Lagrangian flow is the geodesic flow on the flat torus, that is  given by 
 \[ \phi_t(x,v) = ( x+tv \mod \Z^2 ,v ).\]
It follows from straight computations, that $ h_{top}(L,c) =0 $, for all $ c>0$. In this example, the corresponding  Hamiltonian flow is integrable. For an integrable Hamiltonian system on a four dimensional symplectic manifold  under certain regularity assumptions (see \cite{Paternain:1991}), 
 the topological entropy of the  Hamiltonian flow  restricted to a regular compact  energy level is vanishes.


In \cite{Schroder:2016}, J. P. Schr\"oder go toward  a partial answer about the integrability reverse claim of the Paternain's theorem\cite{Paternain:1991} which  is false in general as one already knew the long time ago \cite{Katok:1973}. 
Schr\"oder  proved that if the  topological entropy of the  Lagrangian flow on the level above the Ma\~n\'e  critical value is vanishes thus, for all direction $\zeta \in \mathbb{S}^1$, 
there are  invariant Lipschitz graphs  $\mathcal{T}_{\zeta}$ ($\zeta$ with irrational slope), 
 $\mathcal{T}_{\zeta}^{\pm}$( $\zeta$ with rational slope) over $\mathbb{T}^2$, contained in $\{E=e\}$
whose  complement of its union  is a tubular  neighborhood  of $\mathcal{T}_{\zeta}$ and the lifted orbits from $\mathcal{T}_{\zeta}$,  
$\mathcal{T}_{\zeta}^{\pm}$ on universal cover $\mathbb{R}^2$  going to $\infty$, 
i.e heteroclinic orbits.  He use  the gap condition to prove that if in a strip between two neighboring periodic minimizes no foliation by heteroclinic minimizes exists, 
then there are  instability regions which imply  in its turn positive entropy.

\bigskip

The  Ma\~n\'e's {\it  strict critical value} of $L$,   is the real number $ c_0(L) $ given by 
\begin{equation}
\label{c_0(L)} 
c_0(L)=\inf\{k\in\mathbb{R}:\,A_{L+k}(\gamma)\geq 0, \ \mbox{for all contractible closed curve on  }  \T 
\ \}.
\footnote{On arbitrary  closed manifold,  this equality defines the universal  critical value $c_u(L)$. The value $ c_0(L)$ is the infimum value of $ k\in \R$ such that the $L+k$ action is positive on the set of all closed curves tha are homologous to zero. Of course, $  c_u(L) \leq c_0(L)$, and  $ c_u(L)=c_0(L)$ for systems on $\T$.  } 
 \end{equation}

Here   we prove that

\begin{thm}\label{main-thm}
 Let $L:T\T\rightarrow \mathbb{R}$ a  Tonelli Lagrangian and let  $c>c_0(L)$. Suppose that the Lagrangian flows restricted to a energy level $ E_L^{-1}(c)$  satisfies:
 \begin{enumerate}
\item
all closed orbits  in $E_L^{-1}(c)$ are hyperbolic or elliptic, and
\item
all heteroclinic intersections in $E_L^{-1}(c)$ are transverse.
\end{enumerate} 
Then $ h_{top}(L,c)>0 $.
\end{thm}

Let $ C^r(\T) $ be the set of potentials $ u:\T \to \R $ of class $ C^r $ endowed whit the $C^r$-topology.
We recall that a subset $ \mathcal O  \subset  \m C^r (\T) $ is called  {\it residual} if it contains a countable intersection of open and dense subsets. 
 In \cite{Oliveira:2008}, E. Oliveira proved a version of the
Kupka-Smale Theorem  for  Tonelli Hamiltonian  and  Lagrangian systems on  any closed surfaces.  
 More precisely, it follows from \cite{Oliveira:2008} that,  for each $ c\in \R $, there exists a residual set 
$ \m{KS} (c) \subset \m C^r(\T) $ such that every  Hamiltonian  
$ H_u=H + u $, with  $ u \in  \m{KS}$, satisfies the Kupka-Smale propriety, i.e, all closed orbit with energy $ c$ are hyperbolic  or elliptic and all heteroclinic intersections  are transverse on $ E_L^{-1}(c)$.
See also  \cite{RR:2011}, where L. Rifford and R. Ruggiero proved  the Kupka-Smale Theorem for Tonelli Lagrangian systems on closed  manifolds of any dimension.

So, if we take   the residual subset $ \m{KS} (c) \subset \m C^r(\T) $ given by the Kupka-Smale Theorem, and by the  continuity of the critical values, (cf. Lemma~5.1 in \cite{Contreras_Paternain:2002a}) , we have the following  corollary.

\begin{cor} 
\label{cor}
Given  $c>c_0(L)$ there exist a smooth potential $ u:\T \to \R$ of $ C^r$-norm arbitrarily small (for any $ r \geq  2$) such that  
$ h_{top}(L-u,c) >0$. 
\end{cor}

\section{The Mather and Aubry sets}
\label{Review of the Mather-Aubry Theory}
In this section we  recall the  definitions  of the Mather sets  and Aubry sets   for the case of a (autonomous) Tonelli  Lagrangian on the torus. The Aubry-Mather's theory was  introduced by J. Mather  in \cite{Mather:1991,Mather:1993} for convex, superlinear and time periodic   Lagrangian systems on any closed manifolds. The details and proofs of the main results in this section can be seen in the original works  of J. Mather cited above.

\smallskip

Let us  recall the main  concepts introduced by J. Mather in
\cite{Mather:1991}.  
We denote by $ \mathfrak B(L)$  the set of  all  Borel  probability measures, whit compact support,    that are invariant by the Lagrangian flow of $ L$. 
By duality,  given $ \mu \in \mathfrak B(L) $,
there is a unique homology class $ \rho(\mu) \in H_1(\T,\R)$ such
that 
\begin{equation}
\label{def rho}
 \langle \rho(\mu),[\omega]\rangle  = \int_{T\T} \omega \
d\mu,
\end{equation}
 for any closed 1-form $\omega $ on $ \T$.

Then, the Mather's {\it $\beta$-function} 
 is defined by 
\[ \beta(h) = \inf \left\{
\int_{T\T} L\ d\mu : \ \mu   \in \mathfrak B(L) \mbox{ with }  \rho(\mu)=h\    \ \right\}.\]
 The function  $\beta:H_1(M,\R)\rightarrow \R $ is  convex and superlinear.
  A  measure $ \mu \in \mathfrak B(L) $ that satisfies 
 \[ \int_{T\T} L\ d\mu = \beta(\rho(\mu))\]
 is called a {\it
$\rho(\mu)$-minimizing measure}.

The  Mather's {\it $\alpha$- 
 function} can be defined as the  convex dual (or conjugate) function of $ \beta $, i.e. $ \alpha = \beta^*:
H^1(M,\R)\rightarrow \R $ is given by the so-called Fenchel transformation 
\[ \alpha ([\omega]) = \sup_{h \in H_1(M,\R)}\{<[\omega],h>- \beta(h) \}=-  \inf_{\mu \in \mathfrak B(L)} \left\{  \int_{TM} L-\omega\ d\mu \right\}.\]
 By convex duality, we have that $ \alpha $  is also convex and superlinear, and $\alpha^* = \beta$. 
Moreover, a measure $ \mu_0 $ is $\rho(\mu_0)$-minimizing if and only if there is a closed 1-form  $\omega_0 $, such that $ \int_{TM} L-\omega_0\ d\mu_0 = - \alpha([\omega_0])$. 
Such a class $ [\omega_0] \in H^1(M,\R) $ is  called a {\it subderivative } of $ \beta $  at the point $ \rho(\mu_0)$.

We say that  $ \mu \in \mathfrak B(L)$  is  a  {\it $[\omega]$-minimizing measure} of $ L$ 
if 
\[
 \int_{T\T} L-\omega \ d\mu  
= \min\left\{  \int_{T\T} L-\omega\ d\nu  :  \nu \in \mathfrak B(L)\right\}= -\alpha([\omega]).
\]

Let $\mathfrak M_L([\omega]) \subset \mathfrak B(L) $ be the set of all $[\omega]$-minimizing measures 
 (it only depends on the cohomology class $[\omega]$).  
The ergodic components  of a $[\omega]$-minimizing measure are also  $[\omega]$-minimizing measures,
so the set $\mathfrak M_L([\omega]) $ is a simplex whose extremal measures are ergodic $[\omega]$-minimizing measures. In particular  $\mathfrak M_L([\omega]) $ is a compact subset of $ \mathfrak B(L)$ with the weak$^*$-topology. 

 For each $ [\omega] \in  H^1(\T,\R)$, we define the {\it Mather set  of cohomology class $[\omega]$} as: 
\[  \m{\widetilde M}_L([\omega])= \overline{\bigcup_{\mu \in \mathfrak M_L ([\omega])} \supp( \mu )} .\]  
We set $ \pi(\m{\widetilde M}_L([\omega]))={\m M_L}([\omega])$, and call it the {\it projected Mather set}, where $ \pi:T\T \to \T$ denotes the canonical projection. The celebrated {\it Graph Theorem} proved by J. Mather  in \cite[Theorem 2]{Mather:1991}, 
asserts that   $ \m{\widetilde M}_L([\omega]) $ is non-empty, compact,  invariant by the Euler-Lagrange flow  and the map $ \pi|_{\m{\widetilde M}_L([\omega])}:   \m{\widetilde M}_L([\omega]) \to  {\m M_L}([\omega]) $ is a bi-Lipschitz homeomorphism. 
 
 \bigskip
    
   Following J. Mather in \cite{Mather:1993}, for $ t>0$ and $ x,y \in \T$, define 
 the {\it  action potential} for the  Lagrangian deformed by a closed  1-form $\omega $ as: 
 \[  \Phi_\omega(x,y,t)=\inf \left\{   \int_0^t L(\gamma(s),\dot \gamma(s)) - \omega_{\gamma(s)} (\dot \gamma(s))  \ ds \right\}, \]
 where the infimum is taken over all absolutely  continuous curves $  \gamma:[0,t]\to \T $ such that $ \gamma(0)=x$ and $ \gamma(t)=y$. The infimum is in fact a minimum 
 by  Tonelli's Theorem. 

We define the {\it Peierls barrier} for the  Lagrangian $ L-\omega $ as the function $ h_\omega: \T \times \T \to \R$ given by: 
\[h_\omega(x,y) = \liminf_{t \to +\infty} \left\{ \Phi_\omega (x,y,t) + \alpha([\omega]) t\right\} ,\]
and the {\it projected Aubry set for the cohomology class} $ [\omega]\in H^1(\T,\R)$ as 
 \[\m A_L([\omega])=\left\{ \ x \ \in \T: \ h_\omega(x,x) = 0 \ \right\}.
 \]  By symmetrizing $ h_\omega $, we define the  {\it semidistance}   $\delta_{[\omega]} $ on $ \m A_L([\omega])$:
\[ \delta_{[\omega]}(x,y)=  h_\omega(x,y)+ h_\omega(y,x). \]
This function $\delta_{[\omega]}$ is non-negative and satisfies the triangle inequality. 

We define the {\it Aubry set}  (that is also called static set) as   the invariant set 
  \[ \m{\widetilde A}_L([\omega])=\left\{ (x,v) \in T\T : \pi \circ \phi_t(x,v) \in   \m A_L([\omega]), \ \forall \ t \in \R \right\}.
  \] 
By definition  $ \pi(\m{\widetilde A}_L([\omega]))= {\m A_L}([\omega])$. In \cite[Theorem 6.1]{Mather:1993}, J. Mather proved that this set is  compact, 
 $ \m{\widetilde M}_L([\omega]) \subseteq \m{\widetilde A}_L([\omega]) $ and the extension of the graph theorem to the Aubry set, i.e., the 
mappings  $ \pi|_{\m{\widetilde A}_L([\omega])}:   \m{\widetilde A}_L([\omega]) \to  {\m A_L}([\omega]) $ is a bi-Lipschitz homeomorphism.

Finally, we define the {\it Ma\~n\'e set of cohomology class $[\omega]$}, that we denote by $\m{\widetilde N}_L([\omega])$,
as the set  of orbits $ \phi_t(\theta)=(\gamma(t),\dot\gamma(t))\in T\T$ such that, for all $ a<b \in \R $, the  trajectories 
$ \gamma:\R \to \T$  satisfy 
\[   \int_a^b L(\gamma(s),\dot \gamma(s)) - \omega_{\gamma(s)} (\dot \gamma(s)) +\alpha([\omega]) \ ds  =\inf_{t>0} \left\{ \Phi_\omega (\gamma(a),\gamma(b),t) + \alpha([\omega]) t\right\}. \]
These curves that satisfy the above equality  are also called semi-static curves or $[\omega]$-minimizing curves.

\bigskip

Let us now state some important properties and results on these  invariant sets that we going to use in the proof of the Theorem~\ref{main-thm}.

Using the Mather's $ \alpha$-function, we have the following equivalent definition of the 
Ma\~n\'e strict critical value \eqref{c_0(L)} (cf. \cite{Contreras_Iturriaga_Coloquio:1999a})
\[ c_0(L)= \min\left\{ \alpha([\omega]): [\omega] \in H^1(\T,\R)\right\}=-\beta(0).\]

In \cite{Carneiro:1995},  M. J. Carneiro proved  
that  the set $\m{\widetilde M}_L([\omega])$ is contained in the energy level $  E_L^{-1}(\alpha([\omega]))$ and by  the above characterizations (see for example \cite{Contreras_Iturriaga_Coloquio:1999a}), we have that
  \[ \m{\widetilde M}_L([\omega]) \subseteq  \m{\widetilde A}_L([\omega])
\subseteq \m{\widetilde N}_L([\omega]) \subseteq E^{-1}(\alpha([\omega])  .\]

By the graph property, we can  define an  equivalence relation in
the set $ \m{\widetilde A}_L([\omega]) $: two elements $ \theta_1 $ and $
\theta_2 \in \m{\widetilde A}_L([\omega]) $ are equivalent when $ \delta_{[\omega]}(
\pi(\theta_1),\pi(\theta_2))=0$. The equivalence relation breaks $
\m{\widetilde A}_L([\omega]) $ into classes that are called {\it static classes
of L}. Let $ {\bf \Lambda}_L([\omega]) $ be the set of all static classes. We
define a partial order  $ \preceq $ in $ {\bf \Lambda}_L([\omega]) $ by: (i) $
\preceq $ is reflexive and transitive,
 (ii) if there is $ \theta \in \m{\widetilde N}_L([\omega]) $, such that the  $
\alpha$-limit set  $ \alpha (\theta) \subset \Lambda_i$ and the  $
\omega$-limit set $ \omega(\theta)\subset \Lambda_j $, then $
\Lambda_i  \preceq \Lambda_j$. The following theorem was proved by
G. Contreras and G. Paternain in \cite{Contreras_Paternain:2002a}.
\begin{thm}
\label{T-conect-stat}
 Suppose  that the number of static class is finite,
 then given $ \Lambda_i $ and $ \Lambda_j $ in ${\bf \Lambda}_L([\omega])$,
 we have that $ \Lambda_i \preceq \Lambda_j $.
 \end{thm}

Let $ \Gamma \subset T\T $ be an invariant subset. Given $
\epsilon
> 0$ and $ T >0 $, we say that two points $ \theta_1, \theta_2
\in  \Gamma $ are  {\it $ (\epsilon,T)$-connected by chain} in
$\Gamma$, if there is a finite sequence $ \{ (\xi_i, t_i)
\}_{i=1}^n \subset \Gamma \times \R $, such that $ \xi_1=
\theta_1 $, \ $ \xi_n = \theta_2 $, \ $ T< t_i $ and $ \mbox{dist}(
\phi_{t_i}(\xi_1), \xi_{i+1} ) < \epsilon$, for $ i = 1,..., n-1$.
We say that the subset $ \Gamma $ is {\it chain
transitive}, if for all  $ \theta_1, \theta_2 \in \Gamma $ and
for all $ \epsilon >0 $ and $ T>0 $, the points $ \theta_1 $ and $
\theta_2 $ are $(\epsilon,T)$-connected by chain in $\Gamma $.
When this condition holds for $ \theta_1=\theta_2 $, we say that $
\Gamma$ is {\it chain recurrent}. The proof of the following
properties  can be seen  in \cite{Contreras_Delgado_Iturriaga:1997}.
\begin{thm}
\label{T-rec-estat}\
 $\m{\widetilde A}_L([\omega])  $ is chain recurrent.
\end{thm}

The following
theorem was   proved  by Ma\~n\' e in \cite{Mane:1996}. A proof can
be see also in \cite[theorem IV]{Contreras_Delgado_Iturriaga:1997}.
\begin{thm}
\label{T-min-estat} Let $ \mu\in \mathfrak B(L)$. 
 Then $ \mu \in \mathfrak M_L([\omega]) $ if only if
Supp$(\mu) \subset \m{\widetilde A}_L([\omega]) $.
\end{thm}

\smallskip

Finally, for any closed manifold $ M $, we say that  a class  $ h \in H_1(M,\R) $ is a
{\it rational homology   } if there is $ \lambda>0$ such that $ \lambda h \in
i_*H_1(M,\Z) $, where $ i:H_1(M,\Z)\hookrightarrow H_1(M,\R) $
denotes the inclusion. The following proposition was proved by D. Massart    in \cite{Massart:1997}. 
\begin{prop}
\label{p-h-rational} Let $ M $ be a closed and oriented  surface and let $L$ be a Tonelli Lagrangian on $M$.
Let $ \mu \in \mathfrak B(L) $ be a measure $ \rho(\mu)$-minimizing such
that $ \rho(\mu)$ is a rational homology. Then the support of $ \mu $ is a union of closed orbits or fixed points  of the Lagrangian flow.
\end{prop}

\section{Proof of Theorem \ref{main-thm}}
\label{proofs}

In this section, we prove the Theorem \ref{main-thm}.  

Let $L:T\T\rightarrow \mathbb{R}$ and $c>c_0(L)$.
 We assume that the restricted flow
\[ \left. \phi_t \right|_{E_L^{-1}(c)} :  E_L^{-1}(c) \rightarrow E_L^{-1}(c) \]
satisfies the two conditions: 
 \begin{enumerate}
\item [(c1)]
all closed orbits  are hyperbolic or elliptic, and
\item [(c2)]
all heteroclinic intersections  are transverse.
\end{enumerate}

\begin{lem}
\label{lemmaA1} Let $c>c_0(L)$ and let $h_0 \in H_1(\T, \mathbb{R})\approx \R^2   $ a non-zero class. 
Then there are a invariant probability measure $ \mu_0 $ and a closed 1-form $ \omega_0$ with $ \alpha([\omega_0])=c $, such that:
\begin{itemize}
\item[(i)] $ \rho( \mu_o) = \lambda_0 h_0 $ for some $ \lambda_0 \in R$,
\item[(ii)] $\mu_0 \in \mathfrak M_L ([\omega_0])$ 
and therefore     $ \supp( \mu_0) \subset  E_L^{-1}(c)$.
\end{itemize} 
\end{lem}

\begin{proof} Since $ \beta $ is superlinear, we have:
\begin{equation}
\label{eq-sub-lambda}
 \lim_{\lambda \rightarrow  \infty}
\frac{\beta(\lambda h_0)}{|\lambda h_0|}= \infty.
\end{equation}
 Let $  \partial \beta : H_1(\T,\R ) \rightarrow H_1(\T,\R )^* =H^1(\T,\R )$ be the multivalued function such that to each point $ h \in H_1(\T,\R ) $ associates  all subderivatives of  $ \beta $  in the  point $ h $.
This is well known that, since $ \beta $ is finite, then   $ \partial \beta(h) $ is  a  non empty  convex cone for all $ h \in H_1(\T,\R)$, and $ \partial \beta(h) $ is a unique vector if and only if $ \beta $ is differentiable in $ h $   (see for example \cite [Section 23]{Rockafellar:1970}).
We define  the subset
\[ S(h_0)= \bigcup_{\lambda \in \R} \partial \beta (\lambda h_0 ). \]
By (\ref{eq-sub-lambda}) we have that the subset $ S(h_0) \subset H^1(\T,\R) $ is not bounded. Since $ \beta $ is continuous, by the above properties of the multivalued function $ \partial \beta $, we have that $ S(h_0) $ is a convex subset. 
Observe that if $ \omega \in \partial \beta(0) $, then $\alpha([\omega]) = c_0(L) = \min \{\alpha([\delta])\ ;\ \ \delta \in H^1(\T,\R)\}  $ and, by superlinearity of  $ \alpha $, the restriction $ \alpha|_{S(h_0)} $ is not bounded. 
Therefore, by the  intermediate value theorem,  for each $ c \in [c_0(L), \infty) $ there is $ \omega_0 \in \partial \beta (\lambda_0h_0)\subset S(h_0) $, for some $ \lambda_0 \in \R $, such that $ \alpha([\omega_0]) = c $.
Therefore, if $\mu_0\in\mathfrak M(L)$ is a $(\lambda_0h_0)$-minimizing measure, then  $\mu_0 \in \mathfrak M_L ([\omega_0])$,
 and      $ \supp( \mu_0) \subset  E_L^{-1}(c)$ (by  \cite{Carneiro:1995}).
\end{proof}

\smallskip

Let $i:H_1(\T,\mathbb{Z}) \hookrightarrow H_1(\T, \mathbb{R})$ be the inclusion. Recall that $H^1(\T,\mathbb{Z})\approx\mathbb{Z}^2$ and that $H^1(\T,\mathbb{R})\approx\mathbb{R}^2$. Then $\{(0,1),\,(1,0)\}\subset H_1(\T, \mathbb{Z}) $ is a base of $H_1(\T, \mathbb{Z})$. We have that if $\alpha_0,\,\alpha_1$ are two closed curves in $\T$, with $[\alpha_0]=(0,1)$ and $[\alpha_1]=(1,0)$  then $\alpha_0\cap\alpha_1\neq \emptyset$.

We fix  $h_0=(0,1)\in H_1(\T,\mathbb{Z})$.  
By applying Lemma~\ref{lemmaA1}  we obtain a closed $1$-form $\omega_0$ and a $[\omega_0]$-minimizing measure $ \mu_0$ with  support in to the level $ E_L^{-1}(c)$, for which the  rotational vector   $ \rho( \mu_0)= \lambda_0h_0 $  is  a   rational homology class.
 Therefore, by Proposition \ref{p-h-rational}, the support of $\mu_0$ is formed by  the union of closed  orbits of the flow $  \left. \phi_t \right|_{E_L^{-1}(c)}$.

\begin{lem} \label{lem_uniao_finita}
The Mather's set $\m{\widetilde M}_L([\omega_0])$ is the  union of a finite number of closed orbits for the Lagrangian flow of $L$.
\end{lem}

\begin{proof} 
 The Mather's Graph Theorem asserts that the map $\left. \pi\right|_{\m{\widetilde M}_L([\omega_0])}:\m{\widetilde M}_L([\omega_0])\to \T $ is   injective. 
 Hence, if $\theta_1,\,\theta_2:\mathbb{R}\rightarrow \T$ are two distinct closed orbits of $ \phi_t$ contained in $\m{\widetilde M}_L([\omega_0])$,
  then $\gamma_1(t) = \pi \circ \theta_1(t)$ and $\gamma_2(t) = \pi \circ \theta_2(t)$ must be simple closed curves and $[\gamma_1]=n[\gamma_2]\in H_1(\T, \mathbb{Z})$, because otherwise $\gamma_1\cap\gamma_2\neq\emptyset$. 
 
 Since $c=\alpha([\omega_0])>c_0(L)=-\beta(0)$ and $\mathfrak M_L([\omega_0])$ is a compact set, the continuity of the map $ \rho:\mathfrak{B}(L) \to H_1(\T,\R) $  ( c.f. \cite{Mather:1991}) implies that  there are constants  $k,l\in \R$ such that 
\begin{eqnarray}\label{cotrho}
0<k \leq |\rho(\mu)|  \leq l ,\,\qquad\mbox{for all }\mu\in \mathfrak M_L([\omega_0]).
\end{eqnarray}

 By definition of $\m{\widetilde M}_L([\omega_0])$, we have that $\supp(\mu_0)\subset \m{\widetilde M}_L([\omega_0])$.
Let $ \mu_{\gamma}$ be    the ergodic measure 
supported in a closed orbit $\theta(t)=(\gamma(t), \dot \gamma(t)) \subset \supp(\mu_0) $. Since that $\rho(\mu_0)=\lambda_0h_0$ 
and by the linearity of the map 
$ \rho$, we have
that $[\gamma]=n_0h_0$ for some $0\neq n_0\in\mathbb{Z}$.
It follows from \eqref{def rho}, that 
\[ \langle \rho(\mu_\gamma),[\omega]\rangle  = \lim_{t \to \infty}\frac{1}{t} \int_0^t  \omega(\theta(s))\
ds = \frac{1}{|n_0|T} \int_0^{|n_0|T}  \omega(\theta(s)) \ ds = \frac{1}{|n_0|T} \int_\gamma \omega,\]
 for any closed 1-form $\omega $ on $ \T$. Then
\begin{equation}
\label{rho of ergodic}
\rho(\mu_{\gamma})= \frac{[\gamma]}{|n_0|T}= \pm\  \frac{h_0}{T} ,
\end{equation}
where $T>0$ denotes the minimal period of $\gamma$.
Therefore, the period of any  periodic orbit contained in $\supp(\mu_0)$ is bounded. 

We fix a simple closed orbit 
 $\gamma\in\pi(Supp(\mu_0))$. Since  that $[\gamma]=n_0h_0$ we have that the set $C_\gamma=\T-\{\gamma\}$ define an open cylinder. Let $\mu\neq\mu_0$ be  contained in  $ \mathfrak M_L([\omega_0])$.
 The graph property implies that $\supp(\mu)\cap \supp(\mu_0)=\emptyset$. 
Then $\pi(Supp(\mu))\subset C_\gamma$ and  $\rho(\mu)\in i_*(H_1(C_\gamma,\mathbb{R}))\subset H_1(\T,\mathbb{R})$. Therefore, we have that 
\[\mathfrak M_L([\omega_0])\subset\{\mu\in\mathfrak  M(L)\ :\ \rho(\mu)\in\langle h_0 \rangle_\R)\}.\] 

Applying the Proposition~\ref{p-h-rational} in each measure of the set $\mathfrak M_L([\omega_0])$, 
 we conclude  that the set $ \m{\widetilde M}_L([\omega_0])$ is a union of  closed orbits for the Lagrangian flow. Therefore,  
 the same arguments used on the ergodic components of $ \mu_0$ imply that  each ergodic measure in $\m{\widetilde M}_L([\omega_0])$  satisfies the equality \eqref{rho of ergodic}.
  Then  the inequality (\ref{cotrho}) implies that  the period of all periodic orbit  in $ \m{\widetilde M}_L([\omega_0])$ is uniformly bounded. 
  Therefore, it follows from   the compactness of $ E_L^{-1}(c)$ and  condition (c1)  that there is at most a finite number of closed orbits of $\left. \phi_t \right|_{E_L^{-1}(c)}$ in the Mather's set $ \m{\widetilde M}_L([\omega_0])$.

\end{proof}

\begin{rem} 
\label{rm hyperbolic}  It is  well known that action minimizing curves do not contain {\it conjugate points}~\footnote{ Two points $ \theta_1\not=\theta_2$ are said to be conjugate if $ \theta_2= \phi_\tau(\theta_2)$ and $V(\theta_2)\cap d_{\theta_1}\phi_\tau(V(\theta_1)) \not= \{0\}$, where $ V(\theta)= \ker d_\theta \pi $ denotes the vertical bundle .} 
and a proof of this fact can be see in \cite[\S 4]{Contreras_Iturriaga:1999}. 
So,  by Proposition A in \cite{Contreras_Iturriaga:1999}, for each  $ \theta \in  \m{\widetilde M}_L([\omega_0])$  there exists the Green bundle along of $\phi_t(\theta)$,
 i.e. there is a $ \phi_t$-invariant bi-dimensional subspace  $ \mathbb{E}(\phi_t(\theta)) \subset T_{\phi_t(\theta)} E^{-1}(c) \cong \R^3 $, for all $t \in \R$.    
  Therefore the linearized Poincar\'e   map on  $ \theta $ has an  invariant one-dimensional subspace, so the periodic orbit $ \phi_t(\theta)$  can not be elliptic.
\end{rem}

\smallskip

By Lemma \ref{lem_uniao_finita},  let $\gamma_i:\mathbb{R}\rightarrow\T$, with $i=1,\cdots,n$, be a closed curves  such that 
\[ \m M_L([\omega_0])  =\displaystyle\bigcup_{i=1}^n \gamma_i .\] 
Since $\supp(\mu_0)\subset \m{\widetilde M}_L([\omega_0])$, we have that $[\gamma_i]=n_0h_0=(0,n_0)\in H_1(\T,\mathbb{Z})$, for all $i\in\{1,\cdots,n\}$.


  Let $\m{\widetilde A}_L([\omega_0]) $ the Aubry set corresponding to the cohomology class $ [\omega_0] $  and let ${\bf \Lambda}_L([\omega_0])$ be the set of all static classes. We recall that 
  \[ \m{\widetilde M}_L([\omega_0]) \subseteq  \m{\widetilde A}_L([\omega_0]) .\]

We have that  either 
 $ \m{\widetilde M}_L([\omega_0]) \neq  \m{\widetilde A}_L([\omega_0])$ or $ \m{\widetilde M}_L([\omega_0]) =  \m{\widetilde A}_L([\omega_0])$.

\smallskip

If $ \m{\widetilde M}_L([\omega_0]) \neq  \m{\widetilde A }_L([\omega_0])$, 
for each $\theta \in   \m{\widetilde A}_L([\omega_0])\backslash  \m{\widetilde M}_L([\omega_0])$, 
by the graph property of the Aubry set $\m{\widetilde A}_L([\omega_0])$, 
the curve  $\gamma_{\theta}:=\pi\circ\phi_t(\theta):\mathbb{R}\rightarrow \T $ has no self-intersection points and  $\gamma_{\theta}\cap \m M_L([\omega_0])=\emptyset $.
 Moreover, by Theorem~\ref{T-min-estat}, we have that the $\alpha$-limit and $\omega$-limit sets of $\theta$ are contained in the Mather's set $\m{\widetilde M}_L([\omega_0])$. Since a curve on $\T$, 
 that accumulates in positive time to more than one closed curve, must have self-intersection points, 
 we have that $\omega(\theta)$ is a single closed orbit. 
  By the same arguments, we have that $\alpha(\theta)$ is a single closed orbit. 
 By Remark~\ref{rm hyperbolic} and condition (c2), the orbit $\phi_t(\theta)$
 is in the  transverse  intersections  of a unstable manifold with  a stable manifold of hyperbolic closed orbits, i.e., $\phi_t(\theta)$ is a transverse heteroclinic orbit.
Certainly, if $ \m{\widetilde M}_L([\omega_0])$ is an unique closed orbit, then $\phi_t(\theta)$ is a transverse homoclinic orbit, 
which implies $h_{top}(L,c)>0$ (see for example \cite[p. 276]{Katok_Hasselblatt:1995}). If the set  $\m{\widetilde M}_L([\omega_0]$ has more than one orbit, 
it follows from the recurrence property (Theorem \ref{T-rec-estat}), that $\theta$ is an $(\epsilon,T)$-chain connected in $\m{\widetilde A}_L([\omega_0])$ for all $\epsilon>0$ and $T>0$,
  i.e., there is a finite sequence $\{(\zeta_i,t_i)\}_{i=1}^k\subset \m{\widetilde A}_L([\omega_0]) \times\R$, such that $\zeta_1=\zeta_k=\theta\,,T<t_i$ and $\mbox{dist}(\phi_{t_i}(\zeta_1),\zeta_{i+1})<\epsilon$, for $i=1,\cdots,k-1$. 
  Since the closed orbits in $\m{\widetilde M}_L([\omega_0])$ are isolated on the torus, 
  we have that for $\epsilon$ small enough, the set $\{\pi(\zeta_i)\}_{i=1}^k\subset {\m A_L}([\omega_0])$ must intersect the interior of each one of the cylinders obtained by cutting the torus along the two curves $\gamma_i,\gamma_j \in {\m A_L}([\omega_0])$, with $1\leq i,\,j\leq n$. 
  Therefore, choosing an orientation on ${\m A_L}([\omega_0])$ and reordering the indexes, we obtain a cycle of transverse heteroclinic orbits.  This implies that $h_{top}(L,c)>0$
  
\bigskip
  
  Now we consider the case of $ \m{\widetilde M}_L([\omega_0]) =  \m{\widetilde A}_L([\omega_0])$.
 Since each static class is a connected set (Proposition 3.4 in \cite{Contreras_Paternain:2002a}), for each $1\leq i\leq n$, 
   we have that  $\Lambda_i=(\gamma_i,\dot{\gamma}_i)$. 
   Hence the number of static classes is  equal to  $n$. 
   
Initially, let us to assume that $\m{\widetilde A}_L([\omega_0]  $ has at least two static class, i.e. $ n\geq 2$.  
Let $\Lambda_1$ and $\Lambda_2 $  be two  distinct static classes.
 Applying the Theorem \ref{T-conect-stat}, we have that 
  $\Lambda_1\preceq\Lambda_2$,
 therefore  there exist classes $\Lambda_1,\dots,\Lambda_j=\Lambda_2$ and  points $\theta_1,\dots, \theta_{j-1} \  \in \m{\widetilde N}_L([\omega_0])$ such that, for all $ 1 \leq i \leq j-1 $,  the $\alpha$-limit set $\alpha(\theta_i)\subset\Lambda_i$ and the $\omega$-limit set $\omega(\theta_i)\subset\Lambda_{i+1}$. Also, we have that $\Lambda_2\preceq\Lambda_1$,then  there exist  $\Lambda_{j+1},\dots,\Lambda_{j+k}=\Lambda_1$
 and  points $\theta_{j+1},\dots, \theta_{j+k-1} \  \in \m{\widetilde N}_L([\omega_0])$ such that, for all $ j+1 \leq i \leq j+k-1 $, we have that  $\alpha(\theta_i)\subset\Lambda_i$ and $\omega(\theta_i)\subset\Lambda_{i+1}$.
  By  Remark~\ref{rm hyperbolic}, the corresponding orbits  $(\gamma_i,\dot{\gamma}_i)=\Lambda_i $  for $ 1 \leq i \leq j+k $  are  hyperbolic closed orbits of the flow $\left. \phi_t\right|_{E_{L}^{-1}(c)}$,  and since that the flow $ \left.\phi_t \right|_{E^{-1}(c)}$ satisfies the condition (c2), we have that for all  $ 1 \leq i \leq j+k $ the orbits $\phi_t(\theta_i) $ are in the  transverse heteroclinic intersections of the unstable manifold of $\Lambda_i$ and the stable manifold of $\Lambda_{i+1}$. Then, we obtain a cycle of transverse heteroclinic  orbits,    in particular we have that $h_{top}(L,c)>0$.

\smallskip

Let us to assume now that $ \m{\widetilde A}_L([\omega_0])$  contains  only one  static class $ \Lambda_1=(\gamma_1,\dot\gamma_1)$. In this case,  we can apply the Theorem B in \cite{Contreras_Paternain:2002a}, it   implies that the stable and unstable manifolds of the  hyperbolic closed orbit $ \Lambda_1$ have transverse homoclinic intersections. Here we will give a proof of this  result   in our particular setting.

Let $\Gamma= 2\Z \times \Z $ be a sublattice  of $ \Z^2 \subset \R^2$. We denote by $ \overline{\T} $ the torus defined using the sublattice $ \Gamma$, it is $\overline{\T}= \R^2 / \Gamma$. Let $p: \overline{\T} \to \T$ be the canonical projection. 
 Note that  $p: \overline{\T} \to \T $ is a  double-covering of $ \T$. Let $ \overline{L}: \overline{\T}\to \R $ be the lift of the Lagrangian $ L $ to 
 $\overline{\T}$.  It is well now that $ c_0 (\overline{L})= c_0(L)$ and we have that
 \[ \m A_{\overline{L}}([\omega_0])= p^{-1}(\m A_L([\omega_0]))= p^{-1}(\Lambda_1)\]   
(cf.  \cite[Lemma 2.3]{Contreras_Paternain:2002a}).  By the construction of the covering $p: \overline{\T} \to \T$, the set $ p^{-1}(\Lambda_1) \subset  \overline{\T}$ has two connected components. So, the Aubry set     $\m{\widetilde A}_{\overline{L}}([\omega_0])$ has two static classes and,  as in the above case, there is  a heteroclinic orbit connecting these classes. This heteroclinic connection is  projected in a transverse homoclinic orbit of the hyperbolic orbit $ \Lambda_1=(\gamma_1,\dot \gamma_1)\subset \T$. It completes the proof.     

We note that the homoclinic orbit that we obtained is not in the Ma\~n\'e set ${\m{ \widetilde N}_L}([\omega_0])$, however it is a projection of a orbit in the Ma\~n\'e set for the corresponding  lifted  Lagrangian 
 to the covering $p: \overline{\T} \to \T$.

\section*{Acknowledgments}
We are grateful to Prof. Alexandre Rocha for the helpful conversations.

\nolinenumbers

\bibliographystyle{amsalpha}

\providecommand{\bysame}{\leavevmode\hbox to3em{\hrulefill}\thinspace}
\providecommand{\MR}{\relax\ifhmode\unskip\space\fi MR }
\providecommand{\MRhref}[2]{%
  \href{http://www.ams.org/mathscinet-getitem?mr=#1}{#2}
}
\providecommand{\href}[2]{#2}

\end{document}